\documentclass[11pt,a4paper]{article}
\usepackage[T1]{fontenc}
\usepackage{lmodern}
\usepackage[margin=27mm]{geometry}
\usepackage{amsmath,amssymb,amsthm,mathtools}
\usepackage{microtype}
\usepackage{xcolor}
\usepackage{comment}
\usepackage[colorlinks=true,linkcolor=blue!50!black,urlcolor=blue!50!black,citecolor=blue!50!black]{hyperref}
\newtheorem{theorem}{Theorem}
\newtheorem{lemma}[theorem]{Lemma}
\newtheorem{proposition}[theorem]{Proposition}
\newtheorem{conjecture}[theorem]{Conjecture}
\theoremstyle{remark}
\newtheorem{remark}[theorem]{Remark}
\newtheorem{problem}[theorem]{Problem}
\newcommand{\C}{\mathbb C}
\newcommand{\Z}{\mathbb Z}

\newcommand{\OO}{\mathcal O}

\newcommand{\HH}{\mathcal H}
\newcommand{\dd}{\,d}
\DeclareMathOperator{\Rea}{Re}

\title{Pseudoconvex domains in complex Euclidean space with finite-dimensional Bergman spaces}
\author{Yuan Yuan}
\date{}
\begin{document}
\maketitle
\vspace{-1.5em}

\begin{abstract}
We construct a pseudoconvex domain in $\C^2$ such that the Bergman space is one-dimensional. Based on this example, for any integers $n \geq 3, d \geq 1$, 
 we further construct a pseudoconvex domain in $\C^n$ such that the Bergman space is $d$-dimensional. 
This disproves Wiegerinck's conjecture. 
\end{abstract}

\section{Introduction}\label{sec:introduction}
For a domain $U\subset\C^n$, the Bergman space is the Hilbert space
\[
 A^2(U)=\left\{f\in\mathcal O(U):\int_U|f|^2\dd V_{2n}<\infty\right\},
\]
where $dV_{2n}$ denotes Euclidean Lebesgue measure. If $U$ is bounded,
all holomorphic polynomials belong to $A^2(U)$, so this space is
infinite dimensional. At the other extreme, $A^2(\C^n)=\{0\}$ by the mean-value inequality. 

In 1984, Wiegerinck considered the problem {\it whether a domain in $\C^n$ can admit a finite-dimensional Bergman space} 
\cite{Wiegerinck}. 
Wiegerinck constructed non-pseudoconvex domains
in $\C^2$ with any positive finite-dimension for the Bergman spaces \cite{Wiegerinck}.
Wiegerinck also
established the zero-or-infinity
dichotomy for the Bergman space on
planar domains \cite{Wiegerinck}. More precisely, if $U$ be a planar domain, then $A^2(U)$ is either equal to $\{0\}$ or of infinitely dimensional. 
These two results isolate
pseudoconvexity as a central issue in higher dimensions. 
With the additional assumption of pseudoconvexity, the problem was stated precisely in \cite{JP05} in 2005 and subsequently became known in the literature as Wiegerinck’s conjecture.
More precisely, \emph{Wiegerinck's conjecture} asserts that for a pseudoconvex domain $U\subset\C^n$, $\dim_\C A^2(U)\in\{0,\infty\}$. 
The importance of Wiegerinck’s conjecture comes from the interaction of
holomorphic convexity, global $L^2$ integrability, and intrinsic complex
geometry. Pseudoconvexity controls holomorphic extension and supports
powerful solutions of the $\bar\partial$ equation. A positive answer would impose a strong general
constraint on unweighted $L^2$ function theory, while a finite-dimensional
example identifies a limit of what pseudoconvexity alone guarantees.

Wiegerinck’s conjecture has been extensively studied, with substantial progress toward an affirmative answer. Engliš established the conjecture for pseudoconvex Reinhardt domains \cite{E07}, and Pflug and Zwonek proved it for pseudoconvex balanced domains in \(\mathbb C^2\) \cite{PZ17}. Sufficient conditions for infinite dimensionality of the Bergman space on Hartogs domains were obtained in \cite{J12, B21}. Further sufficient conditions were established in \cite{D07, PZ07, GHH, BZ20} using classical tools from complex analysis, including \(L^2\) methods and potential theory. More recently, Fornæss and Pal revealed connections between the conjecture and the classical union problem \cite{FP22, FP25}. Analogues of Wiegerinck’s problem on complex manifolds have also been investigated in \cite{S22, GGV, KKS}.

On the other hand, the Bergman space determines the Bergman kernel and, when the associated Hermitian form is positive definite, the Bergman metric.
Several decades after Wiegerinck constructed a Reinhardt domain with a nontrivial finite-dimensional Bergman space, Huang and Li revisited related problems from a geometric perspective, motivated by their study of Bergman metrics with constant positive holomorphic sectional curvature. Without assuming pseudoconvexity, they constructed the first Reinhardt domain in \(\mathbb{C}^2\) whose Bergman metric is well defined and has constant positive holomorphic sectional curvature.
The computations in Wiegerinck's paper \cite{Wiegerinck} and Huang–Li's paper \cite{HL} both rely on the classical method of determining Bergman spaces and kernels on Reinhardt domains through the orthogonality and square-integrability of holomorphic monomials. Apart from  this shared standard framework, however, their domain constructions and integrability arguments differ substantially. In particular, they impose different conditions on the exponents of square-integrable monomials, yielding fundamentally different Bergman geometries.
The Huang–Li construction has inspired much  subsequent work on positively curved Bergman metrics  ( see, \cite{HK26, Y26, BGNX}). For instance, Hayashida and Kamimoto \cite{HK26} obtained further results in higher dimensions, motivated by this line of investigation.

Among the many striking results established by Huang and Li \cite{HL}, we highlight their theorem asserting that a well-defined Bergman metric with constant positive holomorphic sectional curvature forces the Bergman space to be finite-dimensional (Theorem 1.1 in \cite{HL}). This remarkable result reveals a deep connection between the dimension of the Bergman space, curvature, projective immersions, and the uniformization of complex manifolds. It also brings Wiegerinck's conjecture into a new geometric perspective. Motivated by these connections, in their breakthrough paper \cite{HL}, Huang and Li proposed the following conjecture:

\begin{conjecture}[Huang-Li \cite{HL}]
Let $U$ be a pseudoconvex domain in $\C^n$ with well-defined Bergman metric (equivalently, the Bergman space is base-point free and separates holomorphic directions). Then the Bergman metric  on $U$ cannot have positive constant holomorphic sectional curvature. 
\end{conjecture}

By Theorem 1.1 in \cite{HL}, Wiegerinck's conjecture would imply this conjecture of Huang and Li.
Huang and Li subsequently constructed Stein manifolds of complex dimension $n$, embedded in $\C^{n+1}$, whose Bergman metrics are well-defined and have positive constant holomorphic sectional curvature \cite{HL2026}. 

\medskip

In the present paper, we construct, for every \(n \geq 2\), a pseudoconvex domain in \(\mathbb{C}^n\) whose Bergman space is one-dimensional, which in turns yields a pseudoconvex domain in \(\mathbb{C}^N\) with Bergman space of dimension $d$ for any integers $N \geq 3, d \geq 1$.

\begin{theorem}\label{thm:main1}
There exists a pseudoconvex domain $\Omega\subset\C^2$ and a nowhere
vanishing $f_0\in\OO(\Omega)$ with finite $L^2$-norm such that the Bergman space  $ A^2(\Omega)=\C f_0$. 
\end{theorem}

\begin{theorem}\label{thm:main}
For integers $n\geq 3$ and $d\geq 1$, there exists a
pseudoconvex domain $\Omega_{n,d}\subset\C^n$ such that
$ \dim_{\C} A^2(\Omega_{n,d})=d$.
\end{theorem}

These examples disprove Wiegerinck's conjecture in every dimension \(n \geq 2\). Their Bergman spaces are base-point free but fail to separate holomorphic directions, so the aforementioned geometric conjecture of Huang–Li remains open.

\medskip

We conclude this section with the following problem:

\begin{problem}
Let $d \geq 2$ be an integer. Does it exist a pseudoconvex domain in $\C^2$ such that the Bergman space is of dimension $d$?
\end{problem}

\section{The construction of the domain in $\C^2$} 
In this section, we first prove  Theorem \ref{thm:main1}. 

\subsection{A pseudoconvex domain in $\C^2$}

Let $\Lambda=\Z+i\Z$ be a lattice and  $E=\C/\Lambda$ be the elliptic curve.
Let $X=E\setminus\{[0]\}$.
Let
$ H(t)=\zeta(t)-\pi t$ be a holomorphic function on $\C\setminus\Lambda$, where $\zeta(t) =\frac1t+
 \sum_{\lambda\in\Lambda\setminus\{0\}}
 \left(\frac1{t-\lambda}+\frac1\lambda+\frac{t}{\lambda^2}\right)$ is the Weierstrass zeta function associated with $\Lambda$. 
The following lemma collects standard facts on the elliptic curve (cf. \cite{F91, AE}).

The next lemma follows from a direct calculation.  

\begin{lemma}
The differential $dt$ on $\C$ descends to a nowhere vanishing holomorphic
  one-form $\omega$ on $E$, and
  \begin{equation}\label{area}
    \int_X\frac{i}{2}\omega\wedge\overline{\omega}=1.
  \end{equation}
\end{lemma}

\begin{lemma}\label{lem:elliptic}
\leavevmode
\begin{enumerate}
  \renewcommand{\labelenumi}{\textup{(\roman{enumi})}}
  \setlength{\itemsep}{0.5em}
  \setlength{\parsep}{0pt}

  \item $H$  satisfies
  \begin{equation}\label{eq:periods}
    H(t+1)=H(t),\quad H(t+i)=H(t)-2\pi i.
  \end{equation}
  Consequently $\psi([t])=\Rea H(t)$ is a smooth harmonic function on $X$.

  \item $X$ admits a proper holomorphic embedding $j$ into $\C^2$,
  and $\tau=|j_1|^2+|j_2|^2$ is a smooth strictly subharmonic exhaustion function.
\end{enumerate}
\end{lemma}

\begin{proof}
(\ref{eq:periods}) is a standard fact on elliptic function (cf. section 7.1 in \cite{AE}). 
It is also a  standard fact that the Weierstrass \(\wp\)-function defined by 
$$ \wp(t)=-\zeta'(t)
 =\frac1{t^2}+
 \sum_{\lambda\ne0}
 \left(\frac1{(t-\lambda)^2}-\frac1{\lambda^2}\right)$$
 is a periodic function and 
satisfies $$ \wp'(t)=-2\sum_{\lambda\in\Lambda}(t-\lambda)^{-3}.$$
It follows that $[t]\mapsto[1:\wp(t):\wp'(t)]$ is a holomorphic embedding from $E$ to 
$\mathbb P^2$ (cf. \cite{F91}). 
Restricting the map to \(X=E\setminus\{[0]\}\) yields a holomorphic embedding $j([t])=(\wp(t),\wp'(t)):  X\to \C^2$. 
Since $\wp(t)=t^{-2}+O(t^2)$ as $t\to0$,  
then \(\lVert j([t])\rVert\to\infty\) and thus $j$ is a proper holomorphic embedding. 
\end{proof}

Since the open Riemann surface $X$ admits a proper holomorphic embedding into $\C^2$, then 
it follows from the celebrated result of Forstneri\v{c}--Wold (cf. Corollary 1.2 in \cite{FW}) that 
$X \times \C$
admits a holomorphic embedding $T$ into $\C^2$ with image $T(X \times \C)$ a Runge domain in $\C^2$..
Define 
\begin{equation*}\label{eq:B}
 B=\{(x,u)\in X\times\C:1<|u|e^{\sqrt{2}\psi(x)}<2\},
\end{equation*}
and let $\Omega=T(B)$. Since $T$ is biholomorphic onto its image,
the restriction $T|_B:B\to\Omega$ is a biholomorphism.

\begin{proposition}\label{lem:psh}
The set $B$ is a connected open set and admits a smooth strictly
plurisubharmonic exhaustion function. Consequently $\Omega=T(B)$ is a
pseudoconvex domain in $\C^2$.
\end{proposition}

\begin{proof}
The map $(x,u)\mapsto(x,e^{\sqrt{2}\psi(x)}u)$ is a homeomorphism
from $B$ onto $X\times\{1<|v|<2\}$, so $B$ is connected.
For $0<s<\log 2$, set
\[
 \chi(s)=\frac1s+\frac1{\log 2-s}.
\]
On $B$ define
\[
 s(x,u)=\log|u|+\sqrt{2}\psi(x),\quad
 \rho(x,u)=\tau(x)+\chi(s(x,u)).
\]
The function $s$ is locally pluriharmonic, and
$\chi''(s)=2s^{-3}+2(\log 2-s)^{-3}>0$. Hence for a tangent vector
$V=(\xi,\eta)$, a direct calculation yields 
\[
 (\partial\bar\partial\rho)(V,\overline V) = (\partial\bar\partial\tau)(\xi,\overline \xi) +\chi''(s)|\partial s(V)|^2,
\]
where 
$(\partial\bar\partial\rho)(V,\overline V) =\sum_{j,k}\frac{\partial^2\rho}{\partial z_j\,\partial\bar z_k} V_j\overline{V_k}.$
If $\xi\ne0$, the first term is positive. If $\xi=0$ and
$\eta\ne0$, then $\partial s(V)=\frac{\eta}{2u}\ne0$. Thus $\rho$
is strictly plurisubharmonic.

We now verify that $\rho$ is an exhaustion function. Fix $M>0$. Since
$\tau\ge0$ and $\chi>0$, the inequality $\rho(x,u)\le M$ implies
$ \tau(x)\le M, \chi(s(x,u))\le M.$
Thus $x$ lies in the compact set ${\mathcal S}_M=\{x\in X:\tau(x)\le M\}$.
Moreover, since $\chi(s)\to+\infty$ as $s\to0$ or $s\to\log 2$,
there exists $\delta>0$, depending only on $M$, such that
$ \delta\le s(x,u)\le\log 2-\delta$
in this sublevel set. The function $\psi$ is bounded on
${\mathcal S}_M$, so the identity $\log|u|=s(x,u)-\sqrt{2}\psi(x)$ yields
positive constants $r_M,R_M$ such that
$r_M\le|u|\le R_M$.
Consequently, every sequence in $\{\rho\le M\}$ has a subsequence
converging to a point $(x,u)\in X\times\C^*$. By continuity,
the limit still satisfies
$\delta\le s(x,u)\le\log 2-\delta$, hence belongs to $B$,
and also satisfies $\rho(x,u)\le M$.
Therefore $\{\rho\le M\}$ is compact, proving that $\rho$
is an exhaustion function.

Since
$\rho\circ(T|_B)^{-1}$ is a smooth strictly plurisubharmonic
exhaustion function of $\Omega$, $\Omega$ is a  pseudoconvex domain in $\C^2$.
\end{proof}

\subsection{The Bergman space}

For a complex surface $M$, let $\HH^2(M,K_M)$ denote the holomorphic
$(2, 0)$-forms with finite canonical norm. More precisely,  in local holomorphic coordinates,
$ \beta=b(z_1,z_2)\,dz_1\wedge dz_2 \in \HH^2(M,K_M)$ if and only if 
$ \|\beta\|^2=i^2\int |b|^2\ dz_1 \wedge d \bar z_1 \wedge dz_2 \wedge d \bar z_2  < \infty.$

\medskip

We first point out the following standard fact in complex analysis. 

\begin{lemma}\label{lem:removable}
If $g$ is holomorphic on a punctured disc and square integrable
near its puncture, then it extends holomorphically across it.
\end{lemma}

\begin{proposition}\label{prop:rankone}
On $B=\{(x,u)\in X\times\C:1<|u|e^{\sqrt{2}\psi(x)}<2\}$, the Bergman space $$ \HH^2(B,K_B)=\C\beta_0, $$
where $ \beta_0=\omega\wedge\frac{du}{u}$ is 
 the nowhere vanishing holomorphic $(2, 0)$-form with finite $L^2$-norm. 
\end{proposition}

\begin{proof}
It is a direct calculation by the Fubini theorem and (\ref{area}) that $\beta_0$ has finite $L^2$-norm. 
Assume $\beta \in  \HH^2(B,K_B)$. Then there exists a unique $F\in\OO(B)$ such that $\beta = F\beta_0$. 
 Lift everything to $\C\setminus\Lambda$ and set $ v=e^{\sqrt{2}H(t)}u.$
The lifted domain becomes
$(\C\setminus\Lambda)\times\{1<|v|<2\}$, and
\begin{equation}\label{eq:formchange}
 dt\wedge\frac{du}{u}=dt\wedge\frac{dv}{v}.
\end{equation}
 In new coordinates, let
\[
 \widehat F(t,v)=F([t],e^{-\sqrt{2}H(t)}v)
              =\sum_{k\in\Z}c_k(t)v^k, 
\]
where each $c_k$ is holomorphic on $\C\setminus\Lambda$.

Denote  the closed fundamental square for the lattice \(\Lambda=\mathbb Z+i\mathbb Z\) by \(Q\), centered at the origin:
\[
Q=\left\{x+iy:\ -\frac12\le x\le\frac12,\quad
                     -\frac12\le y\le\frac12\right\}.
\]
Let $Q_*=Q^\circ\setminus\{0\}$. The quotient map
$t\mapsto[t]$ identifies $Q_*$ with $X$ outside the image of
the boundary of $Q$, which has area zero. We may therefore
compute the norm by integrating over $Q_*$.
Under the change of variables $v=e^{\sqrt{2}H(t)}u$, the fiber
over each $t$ becomes the fixed annulus $1<|v|<2$. Moreover, it follows from 
\eqref{eq:formchange} that 
\[
 F\beta_0=\widehat F(t,v)\,dt\wedge\frac{dv}{v}.
\]
Consequently,
\[
 \|F\beta_0\|^2
 =\int_{Q_*}\int_{1<|v|<2}
       |\widehat F(t,v)|^2\,\frac{dA(v)}{|v|^2}\,dA(t) =\int_{Q_*}\int_1^2\int_0^{2\pi}
       |\widehat F(t,re^{i\theta})|^2\,d\theta\,
       \frac{dr}{r}\,dA(t),
\]
where  $v=re^{i\theta}$. 
For each fixed $t\in Q_*$ and $1<r<2$, the Laurent expansion
converges uniformly on the circle $|v|=r$ and takes the form
\[
 \widehat F(t,re^{i\theta})
 =\sum_{k\in\Z}c_k(t)r^k e^{ik\theta}.
\]
Parseval's identity then yields 
\[
 \int_0^{2\pi}|\widehat F(t,re^{i\theta})|^2\,d\theta
 =2\pi\sum_{k\in\Z}|c_k(t)|^2r^{2k}.
\]
It follows that 
\begin{equation}\label{eq:norm}
 \|F\beta_0\|^2
 =2\pi\int_{Q_*}\int_1^2
       \sum_{k\in\Z}|c_k(t)|^2r^{2k-1}\,dr\,dA(t) =\sum_{k\in\Z}C_k\int_{Q_*}|c_k(t)|^2\,dA(t),
\end{equation}
where
\[
 C_k=2\pi\int_1^2r^{2k-1}\,dr
 =\begin{cases}
   2\pi\log2,&k=0,\\
   \frac{\pi}{k}(2^{2k}-1),&k\ne0.
  \end{cases}
\]
By \eqref{eq:norm}, we have
\[
 \int_{Q_*}|c_k(t)|^2\,dA(t)
 \le \frac{\|F\beta_0\|^2}{C_k}<\infty.
\]
In particular, each $c_k$ is square integrable near the
puncture $t=0$.

On the other hand, it follows from \eqref{eq:periods} that 
\[
 \widehat F(t+1,v)=\widehat F(t,v),\quad
 \widehat F(t+i,e^{-2\pi i\sqrt{2}}v)=\widehat F(t,v).
\]
Uniqueness of Laurent coefficients therefore yields
\begin{equation}\label{eq:characters}
 c_k(t+1)=c_k(t),\quad
 c_k(t+i)=e^{2\pi i\sqrt{2}k}c_k(t).
\end{equation}
Since $\beta = F\beta_0 \in  \HH^2(B,K_B)$, each $c_k$ is 
square integrable near zero by \eqref{eq:norm}. It follows from Lemma~\ref{lem:removable} that $c_k$ extends 
across zero, and thus  extends across every
lattice point by \eqref{eq:characters}. Therefore $c_k$ is entire.
By \eqref{eq:characters}, $c_k$ is bounded on $\C$ and thus a constant by Liouville's theorem. 
If $c_k$ is a non-zero constant, 
then $\sqrt{2}k\in\Z$ by \eqref{eq:characters} again, which forces
$k=0$. Thus every coefficient except for $c_0$ vanishes and $c_0$
is constant. It follows that $F$ is constant.
\end{proof}

\subsection{Proof of Theorem~\ref{thm:main1}}

\begin{proof}[Proof of Theorem~\ref{thm:main1}]
Since $T: B \to \Omega$ is biholomorphic, two Bergman spaces $ A^2(\Omega)$ and
$\HH^2(B, K_B)$ are isomorphic. Write $Z=(Z_1,Z_2)$ for coordinates on $\C^2$.
Then $f_0$ is given by $ f_0 dZ_1\wedge dZ_2=(T^{-1})^*\beta_0$ for some non-vanishing holomorphic function $f_0$ on $\Omega$. It is easy to verify that $f_0 \in L^2(\Omega)$.
\end{proof}

\section{The domain in $\C^n$}\label{sec:all-dimensions}
\subsection{One-dimensional Bergman space}
More generally, it follows from the construction in the dimension 2 that 
 for every
$n\ge2$ there is a pseudoconvex domain $\Omega_n\subset\C^n$ 
with one dimensional Bergman space.

\begin{theorem}\label{thm:all-dimensions}
For every integer $n\ge2$, there exists a pseudoconvex domain
$\Omega_n\subset\C^n$ and a nowhere vanishing holomorphic function
$f_n$ on $\Omega_n$ with  finite $L^2$-norm such that $A^2(\Omega_n)=\C f_n$. 
\end{theorem}

\begin{proof}
The case $n=2$ has been proved. For $n\ge3$, put $m=n-1$ and
retain $X$, $\omega$, $H$, $\psi$, $\tau$, and the embedding $T$
from the two-dimensional construction. Choose real numbers
$a_1=\sqrt2,a_2,\ldots,a_m$ such that
$ 1,a_1,\ldots,a_m$ 
are linearly independent over $\mathbb Q.$
Define
\begin{equation*}\label{eq:Bn}
 B_n=\left\{(x,u_1,\ldots,u_m)\in X\times\C^m:
       1<|u_j|e^{a_j\psi(x)}<2\ \text{for }1\le j\le m\right\},
\end{equation*}
and set $\Omega_n=T_n(B_n)$, where
\begin{equation*}\label{eq:Tn}
 T_n(x,u_1,\ldots,u_m)=\bigl(T(x,u_1),u_2,\ldots,u_m\bigr)
\end{equation*}
is a holomorphic embedding of $X\times\C^m$ into $\C^n$ by applying Corollary 1.2 in \cite{FW}.
The change of variables $v_j=e^{a_j\psi(x)}u_j$ identifies $B_n$
homeomorphically with $X\times\{1<|v|<2\}^m$, so $\Omega_n$
is a domain. With $\chi$ as in the proof of Lemma~\ref{lem:psh},
the same Levi-form and compact-sublevel arguments show that
\[
 \rho_n(x,u)=\tau(x)+\sum_{j=1}^m\chi(s_j(x,u)),\quad
 s_j(x,u)=\log|u_j|+a_j\psi(x),
\]
is a smooth strictly plurisubharmonic exhaustion function of $B_n$.
Thus $\Omega_n$ is pseudoconvex.

Write $\HH^2(B_n,K_{B_n})$ for the square-integrable holomorphic
$n$-forms on $B_n$. The nowhere vanishing form
$ \beta_n=\omega\wedge\frac{du_1}{u_1}\wedge\cdots
                    \wedge\frac{du_m}{u_m}$
has finite $L^2$-norm by Fubini's theorem, as in
Proposition~\ref{prop:rankone}. Every element of
$\HH^2(B_n,K_{B_n})$ is uniquely $F\beta_n$ with $F\in\OO(B_n)$.
On the covering space, use $v_j=e^{a_jH(t)}u_j$ and expand
\[
 \widehat F(t,v)
 =F([t],e^{-a_1H(t)}v_1,\ldots,e^{-a_mH(t)}v_m)
 =\sum_{k\in\Z^m}c_k(t)v^k,
\]
where $ v^k=\prod_{j=1}^m v_j^{k_j}$. 
The coefficients are holomorphic on $\C\setminus\Lambda$, and
\eqref{eq:periods} yields
\begin{equation*}\label{eq:multi-characters}
 c_k(t+1)=c_k(t),\quad
 c_k(t+i)=e^{2\pi i\sum_{j=1}^m a_jk_j}c_k(t).
\end{equation*}
The identity \eqref{eq:formchange} extends to
\[
 dt\wedge\bigwedge_{j=1}^m\frac{du_j}{u_j}
 =dt\wedge\bigwedge_{j=1}^m\frac{dv_j}{v_j}.
\]
The same argument as in the proof of  Proposition~\ref{prop:rankone} implies that every
$c_k$ is square integrable near zero. It follows the same argument that $c_k=0$ for $k \not=0$.
Consequently
\[
 \HH^2(B_n,K_{B_n})=\C\beta_n.
\]

Finally, define the nowhere vanishing holomorphic function $f_n$
on $\Omega_n$ by
\[
 T_n^*(f_n\,dZ_1\wedge\cdots\wedge dZ_n)=\beta_n.
\]
The same biholomorphic change of variables used in dimension two
identifies $A^2(\Omega_n)$ with $\HH^2(B_n,K_{B_n})$.
Hence $f_n$ has finite $L^2$-norm and $A^2(\Omega_n)=\C f_n$.
\end{proof}

\subsection{Bergman space with an arbitrary dimension}

Now we construct the pseudoconvex domain in $\C^n$ with prescribed Bergman space dimension. First, we recall a standard fact in Bergman geometry.

\begin{lemma}\label{lem:product}
Let $Y$ be an open Riemann surface and let $D\subset\C^m$ be a domain
with $A^2(D)=\C f_0$, where $\|f_0\|_{L^2(D)}=1$. Set $ \beta_0=f_0(w)\,dw_1\wedge\cdots\wedge dw_m.$ 
Writing $\pi_Y$ and $\pi_D$ for the projections of $Y\times D$ to $Y$ and $D$, respectively, then the map $ \alpha\mapsto\pi_Y^*\alpha\wedge\pi_D^*\beta_0$ 
is an isometric isomorphism from $\HH^2(Y,K_Y)$ onto
$\HH^2(Y\times D,K_{Y\times D})$.
\end{lemma}

Now we are ready to prove Theorem \ref{thm:main}.

\begin{proof}[Proof of Theorem~\ref{thm:main}]
Fix $n\geq3$ and $d\geq1$. Consider the smooth affine algebraic curve
\[
 X_d=\{(x,y)\in\C^2:y^2=x^{2d+1}-1\}.
\]
Let \(C_d\) be the smooth compactification of \(X_d\), obtained by adjoining the unique point at infinity \(p_\infty\).
It follows from the Riemann--Hurwitz formula that the genus of $C_d$ is given by
\[
 2g(C_d)-2=2(-2)+(2d+2)=2d-2.
\]
It follows from Lemma \ref{lem:removable} that
every square-integrable holomorphic one-form on $X_d$ extends
holomorphically across $p_\infty$. 
Conversely, every holomorphic one-form on
$C_d$ is square integrable on $X_d$.  Thus the following isomorphism holds: 
$$ H^0(C_d,K_{C_d})\cong\HH^2(X_d,K_{X_d}),$$
and 
\begin{equation}\label{eq:curve-space}
 \dim_{\C}\HH^2(X_d,K_{X_d})=d
 \end{equation}
 holds by  the standard equality
$\dim H^0(C_d,K_{C_d})=g(C_d)$ (cf. \cite{F91}).

By Theorem \ref{thm:all-dimensions},
there is a pseudoconvex domain $D\subset\C^{n-1}$ with
$ A^2(D)=\C f_0, \|f_0\|_{L^2(D)}=1.$
Since the inclusion
\[
 \theta:X_d\longrightarrow\C^n,
 \quad \theta(x,y)=(x,y,0,\ldots,0),
\]
is a proper algebraic embedding. It follows from Corollary 1.5 of
\cite{FW} that there exists a holomorphic
embedding
\[
 T:X_d\times\C^{n-1}\longrightarrow\C^n
\]
whose image is a Runge domain. Define
\begin{equation*}\label{eq:domain}
 \Omega_{n,d}=T(X_d\times D).
\end{equation*}
Since $T$ is biholomorphic from $X_d\times D$ to $\Omega_{n,d}$, $\Omega_{n,d}$ is a
connected open subset of $\C^n$.

Since \(X_d\) and \(D\) are Stein, the product \(X_d\times D\) is Stein. Hence \(\Omega_{n,d}\) is Stein and therefore a pseudoconvex domain in $\C^n$.
Moreover, it follows from Lemma~\ref{lem:product} and
\eqref{eq:curve-space} that 
\begin{align*}
 \dim_{\C} A^2(\Omega_{n,d})
 =\dim_{\C}\HH^2(X_d\times D,K_{X_d\times D})=\dim_{\C}\HH^2(X_d,K_{X_d})=d.
\end{align*}
This proves the theorem.
\end{proof}

\begin{remark}
A basis of $H^0(C_d,K_{C_d})$ is given by 
$ \alpha_j=\frac{x^{j-1}\,dx}{y}, 1\leq j\leq d.$
If $Z_1,\ldots,Z_n$ are the coordinates on the target $\C^n$ and
$\beta_0=f_0(w)\,dw_1\wedge\cdots\wedge dw_{n-1}$, the functions
$F_1,\ldots,F_d$ determined by
\[
 (T|_{X_d\times D})^*
   \bigl(F_j\,dZ_1\wedge\cdots\wedge dZ_n\bigr)
 =\pi_{X_d}^*\alpha_j\wedge\pi_D^*\beta_0
\]
form a basis of $A^2(\Omega_{n,d})$.
\end{remark}

\bigskip

{\bf Acknowledgement.}
The author is grateful to Xiaojun Huang for introducing him to this field and for his support and encouragement over the years. He also thanks Ming Xiao for helpful discussions. 
The author was partially supported by Zhejiang
Provincial Natural Science Foundation of China (Grant No. LQKWL26A0201).

\medskip

{\bf AI Statement} ChatGPT 6 Astra was used to improve the exposition of the manuscript and to assist in developing some of the mathematical arguments. The author takes full responsibility for the mathematical content of the paper.

\noindent Yuan Yuan, yuanyuan@westlake.edu.cn, Institute for Theoretical Sciences, Westlake University,\\
Hangzhou 310024, Zhejiang, China\\


\begin{thebibliography}{99}

\bibitem{AE} J. V. Armitage, W. F. Eberlein, Elliptic functions,  London Mathematical Society Student Texts, 67. Cambridge University Press, Cambridge, 2006. xiv+387 pp. ISBN: 978-0-521-78078-0; 978-0-521-78563-1; 0-521-78563-4

\bibitem{BGNX}
S. Bhat, S. Ganguly, A. Nandi amd M. Xiao, {\em Abundance of Bergman metrics with
constant positive holomorphic sectional curvature}, arXiv:2605.17105

\bibitem{BZ20}
Z. Błocki, W. Zwonek, {\em Generalizations of the higher dimensional Suita conjecture and its relation with a problem of Wiegerinck}, J. Geom. Anal. 30 (2020), no. 2, 1259–1270.


\bibitem{B21}
B.~J.~Boudreaux,
\emph{On the dimension of the Bergman space of some Hartogs domains
with higher dimensional bases},
J. Geom. Anal. \textbf{31} (2021), no.~8, 7885--7899.


\bibitem{D07}
Ż. Dinew, {\em The Ohsawa-Takegoshi extension theorem on some unbounded sets}, 
Nagoya Math. J. 188 (2007), 19–30.

\bibitem{E07}
M. Engliš, {\em Singular Berezin transforms}, Complex Anal. Oper. Theory 1 (2007), no. 4, 533–548.

\bibitem{FP22}
J. E. Fornæss, R. Pal, {\em Notes on the short $\C^k$'s}, J. Geom. Anal. 32 (2022), no. 4, Paper No. 133, 22 pp. 

\bibitem{FP25}
J. E. Fornæss, R. Pal, {\em Increasing sequences of complex manifolds with uniform squeezing constants and their Bergman spaces}, 
Math. Ann. 392 (2025), no. 1, 781–796.



\bibitem{F91} O. Forster, Lectures on Riemann surfaces, Translated from the 1977 German original by Bruce Gilligan. Reprint of the 1981 English translation. Graduate Texts in Mathematics, 81. Springer-Verlag, New York, 1991. viii+254 pp. ISBN: 0-387-90617-7

\bibitem{FW}
F.~Forstneri\v{c} and E.~F.~Wold,
\emph{Runge tubes in Stein manifolds with the density property},
Proc. Amer. Math. Soc. \textbf{148} (2020), no.~2, 569--575.


\bibitem{GHH}
A.-K.~Gallagher, T.~Harz, and G.~Herbort,
\emph{On the dimension of the Bergman space for some unbounded domains},
J. Geom. Anal. \textbf{27} (2017), no.~2, 1435--1444.

\bibitem{GGV}
A.-K. Gallagher, P.  Gupta, L. Vivas, {\em On the dimension of bundle-valued Bergman spaces on compact Riemann surfaces}, Indiana Univ. Math. J. 73 (2024), no. 4, 1541–1549.




\bibitem{HK26}
C. Hayashida, J.  Kamimoto, {\em Unbounded Reinhardt domains with finite-dimensional Bergman spaces in $\C^n$},  
J. Geom. Anal. 36 (2026), no. 2, Paper No. 53, 8 pp.


\bibitem{HL}
X.~Huang and S.-Y.~Li, with an appendix by J.~N.~Treuer,
\emph{Bergman metrics with constant holomorphic sectional curvatures},
J. Reine Angew. Math. 822 (2025), 203--220.


\bibitem{HL2026}
X.~Huang and S.-Y.~Li,
\emph{On point separation of Bergman space on Stein manifolds with
constant holomorphic sectional curvature}, arXiv:2602.06813v4

\bibitem{JP05}
M. Jarnicki, P. Pflug, {\em Invariant distances and metrics in complex analysis-revisited}, Dissertationes Math. 430 (2005), 192 pp.

\bibitem{J12}
P.~Jucha,
\emph{A remark on the dimension of the Bergman space of some Hartogs domains},
J. Geom. Anal. \textbf{22} (2012), no.~1, 23--37.

\bibitem{KKS}
L. Koltai, A. Kubasch, R. Szőke, {\em Bergman spaces on algebraic curves}, Math. Z. 313 (2026), no. 2, Paper No. 15, 23 pp.




\bibitem{PZ07} P. Pflug, W. Zwonek, {\em $L^2_h$-domains of holomorphy in the class of unbounded Hartogs domains}, Illinois J. Math. 51 (2007), no. 2, 617–624.

\bibitem{PZ17} P. Pflug, W. Zwonek, {\em $L^2_h$-functions in unbounded balanced domains}, J. Geom. Anal. 27 (2017), no. 3, 2118–2130.


\bibitem{S22} R. Szőke, {\em On a theorem of Wiegerinck}, Anal. Math. 48 (2022), no. 2, 581–587.

\bibitem{Wiegerinck}
J. O.~Wiegerinck,
\emph{Domains with finite dimensional Bergman space},
Math. Z. \textbf{187} (1984), no.~4, 559--562.

\bibitem{Y26}
A. Yamamori, {\em Origin rigidity induced by monomial structure of Bergman spaces on
Reinhardt domains}, J. Funct. Anal. 291 (2026), no. 1, Paper No. 111490, 11 pp.



\end{thebibliography}
\end{document}